\documentclass[12pt, reqno]{amsart}
\usepackage{amsmath, amsthm, amscd, amsfonts, amssymb, graphicx, color}
\usepackage[bookmarksnumbered, colorlinks, plainpages]{hyperref}

\input{mathrsfs.sty}

\newtheorem{theorem}{Theorem}[section]
\newtheorem{lemma}[theorem]{Lemma}

\newtheorem{corollary}[theorem]{Corollary}
\theoremstyle{definition}

\newtheorem{example}[theorem]{Example}

\theoremstyle{remark}
\newtheorem{remark}[theorem]{Remark}
\numberwithin{equation}{section}
\begin{document}

\title [On completely monotonic functions]{On completely monotonic functions}

\author[M. Najafi, A. Morassaei]{Mostafa Najafi$^1$ and  Ali Morassaei$^2$}

\address{$^{1,2}$Department of Mathematics, Faculty of Sciences, University of Zanjan, University Blvd., Zanjan 45371-38791, Iran.}
\email{mostafa.najafi@znu.ac.ir}
\email{morassaei@znu.ac.ir}

\begin{abstract}
Let  $ f:(0,\infty)\rightarrow \Bbb{R} $ be a completely monotonic function. In this paper, we present some  properties of this functions and several new classes of completely monotonic functions. We also give some special functions such that its have completely monotonic condition.
\end{abstract}

\maketitle



\section{Introduction and preliminaries}
A function $ f:(0,\infty)\rightarrow \Bbb{R} $ is said to be \textit{completely monotonic}, if $f $ has derivatives of all orders and satisﬁes
\begin{equation*}\label{cmf}
(-1)^{n} f^{(n)}(x) \geqslant 0,\hspace{1cm} \text{for all}~~ x>0 ~~\text{and}~~ n = 0, 1, 2, 3, \cdots .
\end{equation*} 
This deﬁnition was introduced in 1921 by F. Hausdorﬀ \cite{DVW}, who called such functions ‘total monotone’. Bernstein’s theorem \cite{DVW}, states that $ f $ is completely monotonic if and only if
\begin{equation*}\label{bteq1}
f(x) = \int_{0}^{\infty} e^{-xt} \, \mathsf d \mu(t),
\end{equation*}
where $ \mu $ is a non-negative measure on $ [0,\infty) $ such that the integral converges for all $ x > 0. $
\\
 We can see the importance and applications of this definition in various branches of mathematics and physics, for example, probability theory, numerical and asymptotic analysis, combinatorics and potential theory.

A function $ f:(0,\infty)\rightarrow \Bbb{R} $   is a \textit{Bernstein function} if $  f $ is of class $ \mathcal{C}^{\infty} $, $ f\geqslant 0 $ for all $ x >0 $ and 
\begin{equation}\label{bmf}
(-1)^{n-1} f^{(n)}(x) \geqslant 0,
\end{equation}
for all  $ n\in\Bbb{N} $ \cite{BF}.

A function $ f:(0,\infty)\rightarrow \Bbb{R} $ is said to be \textit{absolutely monotonic} if $f $ has derivatives of all orders and satisﬁes
\begin{equation*}\label{cmf}
 f^{(n)}(x) \geqslant 0,\hspace{1cm} \text{for all}~~ x>0 ~~\text{and}~~ n = 0, 1, 2, 3, \cdots .
\end{equation*} 
A function $ f $ is completely monotonic on an interval $ \mathrm{I} $ if and only if the function $ f(-x) $ is  absolutely monotonic on $ -\mathrm{I} $.

In this article we survey some properties and classifications of completely monotonic functions. 
\section{Main Results}\label{Sec2}
In this section, we present and prove our main results. First, we present some useful lemmas
and theorems  \cite{BF,DVW}.
\begin{theorem} \label{cob}\cite{BF} \label{bfth}
Let $ f $ be a positive function on $ (0,\infty) $. Then the following assertions are equivalent.
\begin{enumerate}
\item[(i)]
f is a Bernstein function.
\item[(ii)]
$ g\circ f $ is a completely monotonic for every $ g $ that  is a completely monotonic function.
\item[(iii)]
$ e^{-uf} $  is a completely monotonic function for every $ u>0 $.
\end{enumerate}
\end{theorem}
\begin{lemma}\label{aoc}\cite{DVW}
Let $ f $ and $ g $ be a absolutely monotonic function and a completely monotonic function, respectively. Then $ f\circ g $ is a completely monotonic function.
\end{lemma}
\begin{lemma}\label{lem2}\cite{DVW}
Let $ f $ and $ g $ be completely monotonic functions. Then $ fg $ is a completely monotonic function. It can be seen from the Leibniz's formula
\begin{equation*}
(fg)^{(n)}=\sum_{k=0}^{n}\binom{n}{k}f^{(k)}g^{(n-k)}.
\end{equation*}
\end{lemma}
\begin{remark}\label{spcmf}
A trivial observation is that if $ f_{1}, f_{2}, \cdots , f_{n}$  be completely monotonic functions, then $ f_{1} f_{2}  \cdots  f_{n} $ is a completely monotonic function. The sum and the product of completely monotonic functions are
also completely monotonic.
\end{remark}
\begin{remark}
A trivial observation is that if $ f_{1}, f_{2}, \cdots , f_{n}$  be completely monotonic functions. Then $ f_{1} f_{2} \cdots f_{n} $ is a completely monotonic function. In the special case, let $ f $   be a completely monotonic function, then $ f^{k}~(k\in\mathbb{N}\cup \left\lbrace 0\right\rbrace ) $  is a completely monotonic function.
\end{remark}
\begin{lemma}
Let $ f $   be a completely monotonic function, but $ f\not\equiv 0 $. Then $ f^{k}\left( k\in\mathbb{Z^{-}}\right) $  is a Bernstein function.
\end{lemma}
\begin{lemma}\label{lemlog}\cite{Boch}
(i) Let
$f: (0,\infty)   \rightarrow  (0,\infty).$
 If  $(-\log f)^{'}$ be a completely monotonic function on $ (0,\infty) $, then $f$ is a completely monotonic function.
 \\
 (ii) Let $ f: (0,\infty) \rightarrow  \Bbb{R} $ and $g: (0,\infty) \rightarrow  (0,\infty)  $. If $ f $ and $ g' $
are completely monotonic, then $ f\circ g $ is also completely monotonic.
\end{lemma}
\begin{lemma}\label{aoc}\cite{DVW}
Let $ f $ and $ g $ be absolutely monotonic functions on  $ (a,b)$, and $ a<g(x)<b $, then $ f\circ g $ is absolutely monotonic function there.
\end{lemma}
\begin{theorem}\label{aoc}\cite{DVW}
If $ f $ is absolutely monotonic function on $ (a,b)$, and $ a<g(x)<b $  completely monotonic function on  $ (a,b)$, then $ f\circ g $ is completely monotonic function there.
\end{theorem}
\begin{corollary}\label{ef}
Let $ f $ be a completely monotone function on $ (0,\infty) $. Then $ e^{f} $  is a completely monotonic function.
\end{corollary}
\begin{proof}
It is enough to see the theorem of  \ref{aoc}. 
\end{proof}
\begin{corollary}
Let $ f $ be a positive function on $ (0,\infty) $. Then $ a^{-tf} $  is a completely monotonic function for every $ a>1 $ and $ t>0 $.
\end{corollary}
\begin{proof}
Since
\begin{equation*}
 a^{-tf}=e^{-tf\ln a}=e^{-(t\ln a)f},
\end{equation*}
So, by using theorem \ref{cob}(iii) we deduce the desired  result.
\end{proof}
\begin{lemma}\label{cob1}
Let $ f $ be a completely monotonic function and let $ g $ be a non-negative function with a completely monotonic derivative.
Then $ f\circ g  $ also is completely monotonic function.
\end{lemma}
\begin{proof}
It suffices refering to the formula for the $ n $-th derivative \cite{Gradshtein}
\begin{equation*}
 \left(f\circ g\right)^{(n)}(x)=\sum_{k=1}^{n}\frac{1}{k!}f^{(k)}(g(x))\sum_{j=0}^{k-1}(-1)^{j}\binom{k}{j}g^{j}(x)\left(g^{k-j}(x)\right)^{(n)}.
\end{equation*}
\end{proof}
With similar way, we get the following theorem.
\begin{theorem}
Let $ f $ and $ g $  be a absolutely monotonic function and a Bernstein  function, respectively.
Then $ f\circ g  $  is a Bernstein function.
\end{theorem}
\begin{lemma}
Let  $ f  $ be a completely monotone function on $ (0,\infty) $  such that $ f \in \mathcal{C}^{\infty}(0,+\infty) $  and  $ f(0)=f\left(0^{+}\right) $, but $ f\not\equiv 1 $. Then $ f\left(x^{\alpha}\right) $ is also completely monotone for all $ 0<\alpha <1 .$
\end{lemma}
\begin{proof}
Since
\begin{equation*}
\frac{d}{dx}\left(x^{\alpha}\right)=\alpha x^{\alpha -1},
\end{equation*}
and assuming $ \alpha >0 $ and $ \beta:=1-\alpha >0 $ we have
\begin{equation*}
\frac{d^{n}}{dx^{n}}\left(x^{-\beta}\right)=(-1)^{n}\beta (\beta +1)\cdots (\beta + n-1)x^{-\beta -n}.
\end{equation*}
So, by using lemma \ref{cob1} we deduce the desired  result.
\end{proof}
\begin{theorem}\label{xc}
Let $ f(x)=x^{x} $ on  $ \left(0,\frac{1}{e}\right) $,
then $f$ is a completely monotonic function.
\end{theorem}
\begin{proof}
Since
\begin{equation*}
\frac{d}{dx}\left(-\log f\right)= \frac{d}{dx}\left(-x\log x\right)=-\log x-1,
\end{equation*}
then using $ h(x)=-\log x -1 $ and Lemma \ref{lemlog}, 
by taking the derivative of $ h $ with respect to $ x $ up to the second order, we have
\begin{align*}
h'(x)&=-\frac{1}{x}<0,\\
h^{''}(x)&=\frac{1}{x^{2}}>0
\end{align*}
and the $  n$-th derivative 
\begin{equation*}
h^{(n)}(x)=\frac{(-1)^{n}}{x^{n}}.
\end{equation*}
In this way, we obtain the desired result.
\end{proof}
\begin{theorem}\label{1xc}
Let $ f(x)=\left(\frac{1}{x}\right)^{\frac{1}{x}} $ on  $ \left(0,\infty\right) $, then $f$ is a completely monotonic function.
\end{theorem}
\begin{proof}
We know that $ g(x)=-\frac{1}{x}\ln x $ is a completely monotonic function. So by using Corollary \ref{ef}, the proof is obviouse.
\end{proof}
Consider the Taylor series expansion of  
\begin{equation*}
f(x)=\sum_{i=0}^{n}a_{i}x^{i},
\end{equation*}
if all coefficients  $ a_{i}  $  are positive and the series converges for all $ x\geqslant 0 $, then $ f $ is an absolutely monotonic function (in that all derivatives are positive on the positive real line).
\begin{corollary}
Let
\begin{equation*}
h(x)=\sum_{i=0}^{n}a_{i}x^{ix},
\end{equation*}
such that all coefficients  $ a_{i} $ are positive and the series converges for all $ x\geqslant 0 .$
Then $ h $ is a completely monotonic function on $ \left(0,\frac{1}{e}\right). $  
\end{corollary}
\begin{proof}
It is enough to see the theorems of  \ref{aoc} and \ref{xc}.
\end{proof}
\begin{corollary}
Let
\begin{equation*}
g(x)=\sum_{i=0}^{n}a_{i}x^{-\frac{i}{x}},
\end{equation*}
such that all coefficients  $ a_{i} $ are positive and the series converges for all $ x\geqslant 0 .$
Then $ g $ is a completely monotonic function on $ \left(0,\infty\right). $  
\end{corollary}
\begin{proof}
It is enough to see the theorems of  \ref{aoc} and \ref{1xc}.
\end{proof}
\begin{remark}\label{f2nf2n+1}
 A trivial observation is that if  $ f $ is a completely monotonic function, then $ f^{(2n)} $ and $ -f^{(2n+1)} $
 are also completely monotonic functions. For example the following functions
 \begin{align*}
f(x)&=\frac{2\ln (x+1)}{x^{3}}-\frac{2}{x^{2}(x+1)}-\frac{1}{x(x+1)^{2}},\\
g(x)&=x^{-\frac{1}{x}-4}\left(\ln^{2}x+(-2x-2)\ln x + 3x+1\right)
\end{align*}
on $ (0,\infty) $ are completely monotonic function.
\end{remark}
\begin{example}
The following elementary functions are immediate examples of completely monotonic functions.
\begin{enumerate}
\item[(i)]
$ e^{-x} $ on $ (0,\infty ) $ is the mother of all completely monotonic functions.
\item[(ii)]
$ e^{\frac{a}{x}} $ on $ (0,\infty ) $ where $ a>0 $, is a completely monotonic function.
\item[(iii)]
$ \ln^{\ln x} x $ on $ \left(0,\frac{1}{e} \right)$ is a completely monotonic function.
\item[(iv)]
 $ \ln\left(1+x\right)^{\frac{1}{x}}   $, $ \frac{1}{\ln(x+1)} $ and $ \left(1+x\right)^{\frac{1}{x}} $ on $ (0,\infty ) $ are   completely monotonic functions.
\item[(v)]
Let $ f $  be a  Bernstein function. Then $ \frac{f(x)}{x}  $ on $ (0,\infty ) $ is a completely monotonic function.
\item[(vi)]
Let $ f $  be a completely monotonic function. Then $ f\left(\ln x\right)  $ on $ (0,\infty ) $ is a completely monotonic function.
\item[(vii)]
 $ \frac{\ln\left(1+\overbrace{\ln\ln \cdots \ln x}^{n-\mathrm{times}}\right)}{\underbrace{\ln\ln  \cdots \ln x}_{n-\mathrm{times}}}  $ on $ (0,\infty ) $ is a completely monotonic function.
  \item[(viii)]
   $ -\ln\left(a+bx^{\alpha}\right)$ on $ (0,\infty ) $,  such that $ a\geqslant 0$, $b> 0 $ and $ 0<\alpha <1 $, is a completely monotonic function.
    \item[(ix)]
  $ -\frac{1}{x}  $ on $ (-1,0) $ is a absouletly monotonic function.
   \item[(x)]
 $ \ln (1+x)  $ and $ \frac{x}{\ln(1+x)} $ on $ (0,\infty ) $ is a Bernstein functions.
\end{enumerate}
\end{example}
\begin{corollary}\label{reBB}
 A trivial observation is that if $ g_{1}, g_{2}, \cdots , g_{2n}$  are Bernstein functions, then $ g_{1} g_{2} \cdots  g_{2n} $ is a completely monotonic function.
 \end{corollary}
 \begin{corollary}\label{reCB}
 A trivial observation is that if $ f_{1} , f_{2}, \cdots , f_{2n}$  are completely monotonic functions; and if $ g_{1},g_{2},\cdots , g_{2n+1}$  be Bernstein functions, then
\begin{equation*}
 f_{1} f_{2} \cdots  f_{2n}  g_{1} g_{2} \cdots   g_{2n+1}
 \end{equation*}
   is a Bernstein function.
\end{corollary}
\begin{lemma}\label{BB}
Let $ f $ and $ g $ be  Bernstein functions. Then $ (-f)^{-g} $ is a completely monotonic functions.
\end{lemma}
\begin{proof}
Since $ f $ and $ g $ are Bernstein functions, so $ -f $ and $ -g $ will be completely monotonic functions. Then,
\begin{equation*}
(-f)^{-g}=e^{-g\ln (-f)}=e^{\ln\left((-f)^{-g} \right) }.
\end{equation*}
So, Theorem \ref{bfth} implies that  $ (-f)^{-g} $  be a completely monotonic function.
\end{proof}
\begin{lemma}\label{1Df}
Let $ f $ be a Bernstein function. Then $ \frac{1}{f}  $ is a completely monotonic function.
\end{lemma}
\begin{proof}
We know that $ g(x)=\frac{1}{x} $ be a completely monotonic function. So by using   Theorem \ref{cob} (ii), the proof is complete.
\end{proof}
\begin{remark}
A trivial observation is that if $ f $  is a completely monotonic function, then $ \frac{1}{f}  $ is a Bernstein function.
\end{remark}
\begin{lemma}\label{-lnb}
Let $ f $ be a Bernstein function. Then $ -\ln f  $ is a completely monotonic function.
\end{lemma}
\begin{proof}
It is enough to see the theorem \ref{cob} (ii).
\end{proof}
\begin{lemma}\label{BC}
Let $ f $ and $ g $ be a Bernstein function and a completely monotonic function, respectively. Then $ f^{g}  $ is a completely monotonic function.
\end{lemma}
\begin{proof}
We have
\begin{equation}\label{glnf}
f(x)^{g(x)}=e^{g(x)\ln f(x)}=e^{-g(x)\left(-\ln f(x)\right)}.
\end{equation}
We know that $ -g $ and $ -\ln f $ be completely monotonic functions. So  Theorem \ref{ef} implies that \eqref{glnf} be a completely monotonic function.
\end{proof}
\begin{lemma}\label{CB}
Let $ f $ and $ g $ be a completely monotonic function and a Bernstein function, respectively. Then $ f^{g}  $ is a completely monotonic function.
\end{lemma}
\begin{proof}
We have
\begin{equation}\label{glnf2}
f(x)^{g(x)}=e^{g(x)\ln f(x)}.
\end{equation}
 So by using Corollaries of \ref{reBB} and \ref{ef} implies that \eqref{glnf2} be a completely monotonic function.
\end{proof}
\begin{lemma}\label{CC}
Let $ f $ and $ g $ be  completely monotonic functions. Then $ f^{g}  $ is a completely monotonic function.
\end{lemma}
\begin{proof}
Put $ f(x)= \frac{1}{h(x)}  $, where $ h $ be a  Bernstein function. By taking the logarithm of the  $ \left(\frac{1}{h(x)}\right)^{g(x)} $, we have
\begin{equation*}
 g(x)\ln\left( \frac{1}{h(x)}\right)=-g(x)\ln h(x).
\end{equation*}
So $ -g\ln h  $ is a completely monotonic function.  Theorem \ref{ef} implies that $ e^{-g\ln h } $ is a completely monotonic function.
\end{proof}
\begin{theorem}\label{CCCBBB}
Let $ f_{1}, f_{2}, \cdots , f_{n}  $ and $ g_{1}, g_{2}, \cdots  , g_{n} (n\geqslant 2) $ be completely monotonic functions and   Bernstein functions, respectively. Then,
\begin{enumerate}
\item[(i)]
 $ f_{1}^{f_{2}^{~.^{.^{.^{~f_{n}}}}}}  $ is a completely monotonic function. 
\item[(ii)]
 $ g_{1}^{g_{2}^{~.^{.^{.^{~g_{n}}}}}}  $ is a completely monotonic function.
 \item[(iii)]
 $ f_{1}^{g_{1}^{~.^{.^{.^{f_{n}^{g_{n}}}}}}}  $ is a completely monotonic function.
\end{enumerate}
\end{theorem}
\begin{proof}
It is enough to see the Lemmas of  \ref{CC} and \ref{BB}.
\end{proof}
\begin{lemma}\label{lem1}
Let 
\begin{equation*}
 f(x)=\sum_{i=0}^{n}\frac{a_{i}}{x^{i}} ,
 \end{equation*}
then $ f $  is a completely monotonic function, where $ a_{i}>0  $ and $ x\in (0,\infty) .$
\end{lemma}
\begin{proof}
It suffices to refer to the formula for the  $ n $-th  derivative 
\begin{equation*}
 f(x)=\left(\frac{1}{x}\right)^{(n)}=\frac{(-1)^{n}n!}{x^{n+1}}.
\end{equation*}
\end{proof}
\begin{theorem}\label{lem2}
Let 
\begin{equation*}
f(x)=\left(  \sum_{i=0}^{n}\frac{a_{i}}{x^{i}}\right) ^{\sum_{i=0}^{n}\frac{a_{i}}{x^{i}}},
\end{equation*}
such that  $ a_{i}\in\mathbb{R}\setminus\left\lbrace 0\right\rbrace $ and $ x\in (0,\infty)  $,  then $ f $  is a completely monotonic function.
\end{theorem}
\begin{proof}
It is enough to see the lemmas of  \ref{BB}, \ref{CC} and \ref{lem1}.
\end{proof}
We are looking for real polynomials $ P_{n}(x)=x^{n}+\sum_{i=0}^{n-1}c_{i}x^{i} ~(n\in \Bbb{N})$ such that
\begin{equation*}
x\rightarrow P_{n}(x)\left(e^{ab} - \left (1 + \frac{a}{x} \right )^{bx+d} \right),
\end{equation*}
is completely monotonic.
\begin{lemma}  \label{lemalzer2}
Let
\begin{equation}\label{eq1alzer2}
f(x)=(x+1)\left(e^{ab} - \left (1 + \frac{a}{x} \right )^{bx+d} \right)
\end{equation}
Then we have
\begin{align}\label{eq2alzer2}
f(x)=\frac{1}{2}a(ab-2d)e^{ab}+\frac{a^{d}}{\pi}\int_{0}^{1}\frac{s^{bs-d}(1-s)^{-bs+d+1}\sin(\pi s)}{x+s}~\mathsf ds.
\end{align}
\end{lemma} 
\begin{proof}
A function $  f$ is holomorphic in the cut plane $ \mathcal{A}=\mathbb{C}\setminus \left( -\infty ,0\right]  $, and satisﬁes  $\mathrm{Im}  f(z)\leqslant 0 $ for $ \mathrm{Im}   z > 0 $ and  $ f(x)\geqslant 0 $ for $ x > 0 $, then $  f$ is a Stieltjes transform, that is, it has the representation
\begin{align*}
f(x)=x_{0}+\int_{0}^{\infty}\frac{\mathsf d\mu(t)}{x+t}
\end{align*}
where $ x_{0}\geqslant 0 $ and $  \mu $ is a non-negative measure on  $ \left[ 0 ,\infty\right) $ with $\int_{0}^{\infty}\frac{d\mu(t)}{x+t}<\infty $\cite{Akhiezer}.

 The constant $ x_{0} $ is given by $ x_{0}=\lim_{x\rightarrow\infty} f(x)$, and $ \mu $ is the limit in the vague topology of measures
\begin{align*}
\mathsf  d\mu(t) = \lim_{y\rightarrow {0^+}} -\frac{1}{\pi}\mathrm{Im} f(-t + iy)~\mathsf d t,
\end{align*}
i.e. 
\begin{align*}
\int \varphi \, \mathsf d\mu = \lim_{y\rightarrow {0^+}} -\frac{1}{\pi} \int \mathrm{Im} f(-t + iy) \varphi(t)~\mathsf d t,
\end{align*}
for all continuous functions $ \varphi $ on $ \mathbb{R} $ with compact support. Let
\begin{align*}
f(z)&=(z+1)\left(e^{ab} - \left (1 + \frac{a}{z} \right )^{bz+d} \right),\\
&=(z+1)\left(e^{ab} -e^{(bz+d)\log \left (1 + \frac{a}{z} \right ) } \right), \quad (z \in \mathcal{A}) 
\end{align*}
where Log denotes the principal branch of the logarithm. We have
\begin{align*} 
f\left( \frac{1}{z}\right)  =e^{ab}\left( 1+\frac{1}{z}\right)\left(1-e^{\frac{(b+dz)\log(1+az)}{z}-ab } \right), 
\end{align*}
so that the series representation 
\begin{align*}
\frac{(b+dz)\log(1+az)}{z} -ab &=\left( -\frac{a^{2}b}{2}+ad\right) z+\left( \frac{a^{3}b}{3}-\frac{ad}{2}\right)z^{2}+  \cdots, \\
&=\sum_{n=1}^{\infty}(-1)^{n}\left( \frac{a^{n+1}b}{n+1}-\frac{ad}{n}\right) \quad (\vert z \vert < 1) .
 \end{align*}
So implies 
\begin{equation}\label{finfinty}
 \lim_{\vert z \vert \rightarrow \infty} f(z) =\frac{1}{2}a(ab-2d)e^{ab}  \quad (z \in \mathcal{A}) .
\end{equation}  
To prove that the harmonic function $ \mathrm{Im} f $ satisﬁes  $ \mathrm{Im} f(z)\leqslant 0 $ for  $  \mathrm{Im} z>0 $, we use the maximum principle for subharmonic functions, cf. \cite{Doob}, and show
that the lim sup of  $ \mathrm{Im} f  $ at all boundary points including inﬁnity is less than or
equal to $ 0 $. From \eqref{finfinty} we conclude that this is true at inﬁnity.
Let $ t\in\mathbb{R} $ and $ z\in\mathbb{C}  $ with $ \mathrm{Im} z>0 $. If $  z $ tends to $  t$, then
\begin{equation*}
\begin{cases}
f(t) & ; t > 0 \\
 (e-1)  & ; t=0 \\
 (1 + t) {\left(e^{ab}  - \exp\left ((bt+d) \log \frac{a+ t}{\vert t \vert} - i\pi t \right )\right)}  & ;  -1 < t <0 \\
 1   & ;  t = -1 \\
(1 + t){\left(e^{ab} - \exp\left ((bt+d) \log \frac{a + t}{t}\right )\right)}  & ;~~ t < -1
\end{cases}
\end{equation*}  
In particular, if $  y $ tends to $ 0^{+} $, then we obtain for $  t\in\mathbb{R} $ 
\begin{equation*}
- \frac{1}{\pi} \mathrm{Im}{ f(- t + iy)} \longrightarrow
\begin{cases}
0 &;  t \leq 0 \quad t \geq 1 \\
\frac{1}{\pi}(1 - t) \exp((-bt+d) \log (\frac{a - t}{t})) \sin(\pi t) &;  0 < t <1
\end{cases}
\end{equation*}
The limit above is uniform for t in compact subsets of the real axis, and therefore
the (continuous) density on the right is the vague limit of the densities on the left. The proof of Lemma \ref{lemalzer2} is complete.
\end{proof}
\begin{lemma} \cite{HA1} \label{lemalzer1}
Let
\begin{equation}\label{eq1alzer1}
g(s) = \frac{a^{d}}{\pi}s^{bs-d}(1-s)^{-bs+d+1}\sin(\pi s), \quad (0 \leqslant s \leqslant 1)
\end{equation}
Then we have
\begin{align*}
& \int_0^1 g(s)~\mathsf d s =-\frac{1}{24}ae^{ab}\left( 12ad^{2}-12\left( a^{2}b+a-2\right)d+ab\left( a\left( 3ab+8\right) -12 \right)  \right)  ,  \\
& \int_0^1 sg(s)~\mathsf d s = -\frac{1}{48}ae^{ab}\left( 12ad^{2}-12\left( a^{2}b+a-2\right)d+ab\left( a\left( 3ab+8\right) -12 \right)  \right) .
\end{align*}
\end{lemma}
\begin{proof}
 Let $  f $ be the function given in \eqref{eq1alzer2}. We set $ x:=\frac{1}{y} $ and obtain
 \begin{small}
\begin{align}\label{e24}
x\left( f(x)-\frac{1}{2}a(ab-2d)e^{ab}\right)&=\frac{e^{ab}- \left (1 + ay \right )^{\frac{b}{y}+d}+y\left(e^{ab}- \left (1 + ay \right )^{\frac{b}{y}+d} \right) -\frac{1}{2}ya(ab-2d)e^{ab}}{y^{2}}\nonumber\\
&:=g(y) .
\end{align}
 \end{small}
The rule of L'Hôpital yields
 \begin{equation}\label{e24a}
\lim_{y\rightarrow 0}g(y)=-\frac{1}{24}ae^{ab}\left( 12ad^{2}-12\left( a^{2}b+a-2\right)d+ab\left( a\left( 3ab+8\right) -12 \right)  \right).
\end{equation}
From \eqref{eq1alzer2}, \eqref{eq2alzer2}, and \eqref{e24} we get
\begin{align*}
\int_0^1 g(s)~\mathsf d s&=\lim_{x\rightarrow \infty}\int_{0}^{1}\frac{x}{x+s}g(s)\mathsf d s\\
&=\lim_{x\rightarrow \infty}x\left( f(x)-\frac{1}{2}a(ab-2d)e^{ab}\right)\\
 &=-\frac{1}{24}ae^{ab}\left( 12ad^{2}-12\left( a^{2}b+a-2\right)d+ab\left( a\left( 3ab+8\right) -12 \right)  \right).
\end{align*}
Using $ g(s) = g(1 - s) $ we obtain
\begin{align*}
\int_0^1 sg(s)~\mathsf d s =& \int_0^{\frac{1}{2}} s g(s)~\mathsf d s + \int_{\frac{1}{2}}^1 sg(s) ~\mathsf d s, \\
=&\int_0^{\frac{1}{2}} sg(s)~\mathsf d s + \int_0^{\frac{1}{2}} (1 - s)g(s)~\mathsf d s, \\
=& \int_0^{\frac{1}{2}} g(s)~\mathsf d s, \\
=& \frac{1}{2} \int_0^1 g(s)~\mathsf d s, \\
=&-\frac{1}{48}ae^{ab}\left( 12ad^{2}-12\left( a^{2}b+a-2\right)d+ab\left( a\left( 3ab+8\right) -12 \right)  \right).
\end{align*}
Using \eqref{eq2alzer2} we get
\begin{align*}
f(x)-\frac{1}{2}a(ab-2d)e^{ab}=\frac{a^{d}}{\pi}\int_{0}^{1}\frac{s^{bs-d}(1-s)^{-bs+d+1}\sin(\pi s)}{x+s}~\mathsf d s.
\end{align*}
\end{proof}
\begin{theorem}
Let $ P_{n}(x)=x^{n}+\sum_{i=0}^{n-1}c_{i}x^{i} ~(n\in \Bbb{N})$  be a polynomial of degree $ n\geqslant 1 $
with real coeﬃcients. The function
\begin{equation*}
x \mapsto P_{n}(x)\left(e^{ab} - \left (1 + \frac{a}{x} \right )^{bx+d} \right),
\end{equation*}
is completely monotonic if and only if $ n=1 $ and $ c_{0}\geqslant \frac{1}{12}a\left(6+ab\left(3+\frac{2}{ab-2d}\right)-6d\right) $.
\end{theorem}
\begin{proof}
 Let
 \begin{equation*}
F_{n}(x)= P_{n}(x)\left(e^{ab} - \left (1 + \frac{a}{x} \right )^{bx+d} \right),~~(x>0).
\end{equation*}
There exists a number $ x_{0}>0 $  such that $  P_{n} $ is positive on $ (x_{0},+\infty) $. If $ F_{n} $  is
completely monotonic, then we have
\begin{small}
 \begin{equation*}
 F'_{n}(x)= P'_{n}(x)\left(e^{ab} - \left (1 + \frac{a}{x} \right )^{bx+d} \right)+ P_{n}(x)\left (1 + \frac{a}{x} \right )^{bx+d}\left(b\log \left(1 + \frac{a}{x}\right)-\frac{a(bx+d)}{x(x+a)}\right)\leqslant 0.
\end{equation*}
\end{small}
We put
 \begin{equation*}
g_{n}(x)=\frac{e^{ab} - \left (1 + \frac{a}{x} \right )^{bx+d} }{\left (1 + \frac{a}{x} \right )^{bx+d}\left(b\log \left(1 + \frac{a}{x}\right)-\frac{a(bx+d)}{x(x+a)}\right)}-x,
\end{equation*}
and
 \begin{equation}\label{limgn}
\lim_{x\rightarrow\infty}g_{n}(x)=\frac{1}{12}a\left(6+ab\left(3+\frac{2}{ab-2d}\right)-6d\right),
\end{equation}
so for $ x>x_{0} $ this is equivalent to
 \begin{equation}\label{xdpp}
x\frac{ P'_{n}(x)}{ P_{n}(x)}=\left(1+\frac{c_{0}}{x}\right)^{-1}\leqslant \left(1+\frac{g_{n}(x)}{x}\right)^{-1}.
\end{equation}
If we take a limit from the \eqref{xdpp}
 \begin{equation*}
n=\lim_{x\rightarrow\infty}x\frac{ P'_{n}(x)}{ P_{n}(x)}\leqslant \lim_{x\rightarrow\infty}\left(1+\frac{g_{n}(x)}{x}\right)^{-1}=1.
\end{equation*}
This implies $ n=1 $. Hence, $  P_{n}(x)=x+c_{0} $ , so that \eqref{xdpp} yields
\begin{equation}\label{gc0}
g_{n}(x)\leqslant c_{0}\qquad\qquad\text{for all suﬃciently large $  x$}.
\end{equation}
From \eqref{limgn} and \eqref{gc0} we conclude that $ c_{0}\geqslant\frac{1}{12}a\left(6+ab\left(3+\frac{2}{ab-2d}\right)-6d\right) $.\\
Let
\begin{equation} \label{eq 13 1}
f_{c_{0}}(x) = (x + c_{0}) \left(e^{ab} - \left (1 + \frac{a}{x} \right )^{bx+d}\right),   ~~~~~~c:=\frac{1}{12}a\left(6+ab\left(3+\frac{2}{ab-2d}\right)-6d\right)
\end{equation}
We have
\begin{align*}
f_{c_{0}} = & (x+c_{0}) \left(e^{ab} - \left (1 + \frac{a}{x} \right )^{bx+d}\right), \\
=& (x +c_{0})e^{ab} - (x +c_{0}) \left (1 + \frac{a}{x} \right )^{bx+d}, \\
=& \left (x +c +c_{0} -c\right )e^{ab} - \left (x +c +c_{0} -c\right )\left (1 + \frac{a}{x} \right )^{bx+d} ,\\
= &\left (x + c\right )e^{ab} + \left (c_{0} - c\right )e^{ab} - \left (x + c \right )\left (1 + \frac{a}{x} \right )^{bx+d} - \left (c_{0} - c\right )\left (1 + \frac{a}{x} \right )^{bx+d}, 
\end{align*}
We multiply the above expression by $ (x + c) $ in this case 
\begin{align*}
(x + c) f_{c_{0}}(x) =  &\left (x +c\right )^2e^{ab} + \left (x +c\right )\left (c_{0} -c\right )e^{ab} \\
&  - \left(x + c \right)^2\left (1 + \frac{a}{x} \right )^{bx+d} - \left (x +c\right )\left (c_{0} - c\right )\left (1 + \frac{a}{x} \right )^{bx+d}, \\
=&\left (x +c\right )\underbrace{\left (\left (x + c\right ) \left(e - \left (1 + \frac{1}{x}\right )^x \right )\right)}_{f_{c}(x)} \\
 & + \left (c_{0} -c\right )\underbrace{\left (\left (x + c\right ) \left(e^{ab} - \left (1 + \frac{a}{x} \right )^{bx+d}\right)\right)}_{f_{c}(x)} \\
=& \frac{1}{x +c} \left (\left (x + c \right ) f_{c}(x) + \left (c_{0} - c\right ) f_{c}(x)\right ),
\end{align*}
If $c_{0}\ge c $, then $ x\rightarrow\frac{c_{0} - c }{x + c } $ is completely monotonic.
\\
Applying Remark \ref{spcmf} we get: if $ f_{c} $  is completely monotonic, then the same is true for
$ f_{c} $ with $ c_{0} \geqslant c $.
We prove now that $ f_{c} $ is completely monotonic. Let g be deﬁned in \ref{eq1alzer1}. Applying
\begin{equation}\label{eqlapxn}
\frac{1}{(x + s)^n} = \frac{1}{(n - 1)!} \int_0^\infty e^{-(x+s)t}t^{n - 1}~\mathsf d t, \quad (x > 0, s \geq 0, n = 1, 2, \ldots)
\end{equation}
with $ n = 1 $ we obtain
\begin{equation}
 \frac{1}{x + s} = \int_0^\infty e^{-xt} e^{-st}~\mathsf d t.
\end{equation}
So we have
\begin{align}\label{eqgs2}
\int_0^1 \frac{g(s)}{x + s}~\mathsf d s =& \int_0^1 \int_0^{\infty}  e^{-xt} e^{-st} g(s)~\mathsf d t~\mathsf d s, \nonumber \\
=&  \int_0^{\infty}  e^{-xt} \left (\int_0^1  e^{-st} g(s)~\mathsf d s\right )~\mathsf d t, \nonumber\\
=&\int_0^\infty e^{-xt} h(t)~\mathsf d t,
\end{align}
where
\begin{equation}
h(t) = \int_0^1 e^{-st} g(s)~\mathsf d s.
\end{equation}
From \eqref{eq2alzer2} and \eqref{eqgs2} we get
\begin{equation}\label{eqfc02}
f_{c_{0}}(x) = \frac{1}{2}a(ab-2d)e^{ab} + \int_0^\infty e^{-xt} h(t)~\mathsf d t.
\end{equation}
and
\begin{equation*}
f_{c}(x) = \frac{x +c}{x + 1} f_{c_{0}}(x) = \frac{1}{2}a(ab-2d)e^{ab} \left (\frac{x +c}{x + 1}\right ) + \frac{x +c}{x + 1} \int_0^\infty  e^{-xt} h(t)~\mathsf d t.
\end{equation*}
We set $ \alpha +c=1 $, so
\begin{equation}\label{alfa}
\alpha =\frac{12-a\left( 6+3ab+\frac{2ab}{ab-2d}-6d\right) }{12}.
\end{equation}
Therefore
\begin{equation*}
=  \frac{1}{2}a(ab-2d)e^{ab}\left (\frac{x +c + \alpha  - \alpha }{x + 1}\right ) + \frac{x +c + \alpha  - \alpha }{x + 1} \int_0^\infty e^{-xt} h(t)~\mathsf d t.
\end{equation*}
 We know that 
 \begin{equation}\label{Laplace1}
 \frac{1}{x + 1}= \int_0^\infty e^{-xt} e^{-t}~\mathsf d t.
 \end{equation}
Using \eqref{eqfc02}, \eqref{eqlapxn} with $ s = n = 1 $, and the convolution theorem for Laplace transforms \eqref{Laplace1} we obtain
\begin{small}
\begin{align}\label{eqfc3}
f_{c}(x) =&   \frac{1}{2}a(ab-2d)e^{ab}\left (1 - \alpha \left (\frac{1}{x + 1}\right )\right ) + 1 - \alpha \left (\frac{1}{x + 1}\right ) \int_0^\infty e^{-xt} h(t)~\mathsf d t,\nonumber \\
=& \frac{1}{2}a(ab-2d)e^{ab}\left(1 - \alpha  \int_0^\infty e^{-xt} e^{-t}~\mathsf d t\right) + \left(1 - \alpha  \int_0^\infty e^{-xt} e^{-t}~\mathsf d t\right)\int_0^\infty e^{-xt} h(t)~\mathsf d t,\nonumber  \\
=& \frac{1}{2}a(ab-2d)e^{ab} -  \frac{1}{2}a(ab-2d)e^{ab}\alpha \int_0^\infty e^{-xt} e^{-t}~\mathsf d t + \int_0^\infty e^{-xt} h(t)~dt \nonumber\\
&-  \alpha  \int_0^\infty e^{-xt} e^{-t}~\mathsf d t \int_0^\infty e^{-xt} h(t)~\mathsf d t,\nonumber  \\
=&  \frac{1}{2}a(ab-2d)e^{ab} + \int_0^\infty e^{-xt} u(t)~\mathsf d t,
\end{align}
\end{small}
where
\begin{align*}
u(t) = h(t) - e^{-t}\left(\frac{1}{2}a(ab-2d)e^{ab}\alpha + \alpha \int_0^t e^s h(s)~\mathsf d s\right).
\end{align*}
In order to prove that $ u $ is positive on $ (0,\infty ) $ we set $ v(t) = e^{t} h(t) $ and $ w(t) =
e^{t} u(t) $. Then we have
\begin{align*}
w(t)= v(t) -\left( \frac{1}{2}a(ab-2d)e^{ab}\alpha + \alpha \int_0^t e^s h(s)~\mathsf d s\right).
\end{align*}
Let $ t>0 $. Diﬀerentiation gives
\begin{align}\label{wet}
w^{\prime}(t) = &v^{\prime}(t) - \alpha v(t),\nonumber \\
= & e^t h(t) + e^t h^{\prime}(t) - \alpha \left (e^t h(t)\right ),\nonumber \\
= & e^t \left ( h(t) + h^{\prime}(t) - \alpha  h(t)\right ),\nonumber \\
= & e^t \left (ch(t) + h^{\prime}(t)\right ).
\end{align}
Applying \eqref{wet} we obtain
\begin{align}\label{eqwet}
w^{\prime}(t)e^{-t} = &c\int_0^1 e^{-st}g(s)~\mathsf d s + \int_0^1 -s e^{-st}g(s)~\mathsf d s, \nonumber\\
= & \int_0^1 e^{-st}\left(c - s\right) g(s)~\mathsf d s\nonumber \\
= & \int_0^{c} e^{-st}\left(c - s\right)g(s)~\mathsf d s + \int_{c}^1 e^{-st}\left(c - s\right)g(s)~\mathsf d s \nonumber\\
\geq & e^{-c t} \int_0^{c} \left(c - s\right)g(s)~\mathsf d s + e^{-c t} \int_{c}^1  c g(s)~\mathsf d s, \nonumber\\
= & e^{-c t}\int_0^1 \left(c - s\right)g(s)~\mathsf d s.
\end{align}
Hence, we have
\begin{align*}
\int_0^1 \left(c - s\right)g(s) \mathsf d s= \frac{-2c+1}{48} ae^{ab}\left( 12ad^{2}-12\left( a^{2}b+a-2\right)d+ab\left( a\left( 3ab+8\right) -12 \right)  \right),
\end{align*}
so that \eqref{eqwet} reveals that $ w_{0} $ is positive on $ (0,\infty) $. Hence, we get for  $ t > 0 $ 
\begin{align*}
w(t) > w(0) = u(0) &=v(0)- \frac{1}{2}a(ab-2d)e^{ab}\alpha \\
&= h(0) -  \frac{1}{2}a(ab-2d)e^{ab}\alpha \\
&= \int_0^1 g(s)~\mathsf d s -  \frac{1}{2}a(ab-2d)e^{ab}\alpha = 0.
\end{align*}
This implies that $  w$ and $  u$ are positive on $ (0,\infty) $. From the integral representation \eqref{eqfc3} we conclude that  $ f_{c} $ is completely monotonic.
\end{proof}
\bibliographystyle{amsplain}

\begin{thebibliography}{99}
\bibitem{Akhiezer}
N.I. Akhiezer, \textit{The Classical Moment Problem and Some Related Questions in Analysis}. - English translation, Oliver and Boyd, Edinburgh, 1965.
\bibitem{HA1}
H. Alzer,  C. Berg, Some classes of completely monotonic functions, \textit{Ann. Acad. Sci. Fenn. Math}. \textbf{27} (2002), 445–460.
\bibitem{HA2}
H. Alzer,  C. Berg,  Some classes of completely monotonic functions. II. \textit{ Ramanujan J. } \textbf{11} (2006), no. 2, 225-248
\bibitem{Boch}
 S. Bochner, \textit{Harmonic analysis and the theory of probability}, Univ. of California Press, Berkeley-Los Angeles,
1960.
\bibitem{Doob}
J.L. Doob, \textit{Classical Potential Theory and its Probabilistic Counterpart.} - Springer- Verlag, New York, 1984.
\bibitem{Gradshtein}
I.S. Gradshtein and I.M. Ryzhik, \textit{Tables of integrals, Sums, Series and products}. 7th
Edition, Academic Press, Inc., 1994.
\bibitem{Kim}
C.H. Kimberling,  A probabilistic interpretation of complete monotonicity,  \textit{Aequat. Math.} \textbf{10} (1974), 152–164.
\bibitem{JBR}
J.B. Rosser, The complete monotonicity of certain functions derived from completely monotonic functions. \textit{Duke Math. J.} \textbf{15} (1948), 313–331.
\bibitem{BF}
R. L. Schilling, R. Song, Z. Vondraček, \textit{Bernstein functions: Theory and applications}.  2012.
\bibitem{DVW}
D.V. Widder, \textit{The Laplace transform}. Princeton Univ. Press, Princeton, NJ , 1941.
\bibitem{Feng}
S. Zhang, C. Chen, F. Qi, On a completely monotonic function. \textit{Math. Practice Theory,} \textbf{36} (2006) , 236-238.
\end{thebibliography}

\end{document}